\documentclass[12pt]{amsart}
\usepackage{tikz-cd}
\usepackage{amsmath}
\usepackage{fourier}
\usepackage{amssymb}
\usepackage{amscd}
\usepackage{amsthm}
\usepackage[centertags]{amsmath}
\usepackage{amsfonts}
\usepackage{newlfont}
\usepackage{graphicx}
\usepackage{amsfonts, amssymb}
\usepackage{mathrsfs}
\usepackage{latexsym}
\usepackage{tikz}
\usepackage{verbatim}
\usepackage[all]{xy}
\usepackage{enumitem}
\usepackage[colorlinks=true,linkcolor=colorref,citecolor=colorcita,urlcolor=colorweb]{hyperref}
\definecolor{colorcita}{RGB}{21,86,130}
\definecolor{colorref}{RGB}{5,10,177}
\definecolor{colorweb}{RGB}{177,6,38}

\newtheorem{theorem}{Theorem}[section]

\newtheorem{proposition}[theorem]{Proposition}
\newtheorem{corollary}[theorem]{Corollary}
\newtheorem{lemma}[theorem]{Lemma}

\theoremstyle{definition}
\newtheorem{remark}[theorem]{Remark}

\theoremstyle{remark}

\newcommand{\dis}{\displaystyle}

\usepackage{mathtools}

\DeclareMathOperator{\re}{Re}

\DeclareMathOperator{\id}{\mathrm{id}}

\newcommand\restrict[1]{\raisebox{-.5ex}{$\vert$}_{#1}}

\begin{document}
\title{The projection constant for the trace class}

\author[Defant]{A.~Defant}
\address{%
Institut f\"{u}r Mathematik,
Carl von Ossietzky Universit\"at,
26111 Oldenburg,
Germany}
\email{defant@mathematik.uni-oldenburg.de}

\author[Galicer]{D.~Galicer}
\address{Departamento de Matem\'{a}tica,
Facultad de Cs. Exactas y Naturales, Universidad de Buenos Aires and IMAS-CONICET. Ciudad Universitaria, Pabell\'on I
(C1428EGA) C.A.B.A., Argentina}
\email{dgalicer$@$dm.uba.ar}

\author[Mansilla]{M~Mansilla}
\address{Departamento de Matem\'{a}tica,
Facultad de Cs. Exactas y Naturales, Universidad de Buenos Aires and IAM-CONICET. Saavedra 15 (C1083ACA) C.A.B.A., Argentina}
\email{mmansilla$@$dm.uba.ar}

\author[Masty{\l}o]{M.~Masty{\l}o}
\address{Faculty of Mathematics and Computer Science, Adam Mickiewicz University, Pozna{\'n}, Uniwersytetu \linebreak Pozna{\'n}skiego 4, 61-614 Pozna{\'n}, Poland}
\email{mieczyslaw.mastylo$@$amu.edu.pl}

\author[Muro]{S.~Muro}
\address{FCEIA, Universidad Nacional de Rosario and CIFASIS, CONICET, Ocampo $\&$ Esmeralda, S2000 Rosario, Argentina}
\email{muro$@$cifasis-conicet.gov.ar}

\date{}

\thanks{The research of the fourth author was supported by the National Science Centre (NCN), Poland, Grant no. 2019/33/B/ST1/00165; the second, third and fifth author were supported by CONICET-PIP 11220200102336 and PICT 2018-4250. The research of the fifth author is additionally supported by ING-586-UNR}

\begin{abstract}
We  study the projection constant of the space of operators on $n$-dimensional Hilbert spaces, with the trace norm, $\mathcal S_1(n)$. We show an
integral formula for the projection constant of  $\mathcal S_1(n)$; namely
$
\boldsymbol{\lambda}\big(\mathcal S_1(n)\big) = n \int_{\mathcal U_n} \vert \text{tr}(U) \vert \,dU \,,
$
where the integration is with respect to the Haar
probability measure on the group $\mathcal U_n$ of unitary operators.  Using a probabilistic
approach, we derive the limit formula
$
\dis\lim_{n\to \infty}
\boldsymbol{\lambda}\big(\mathcal S_1(n)\big)/n
= \sqrt{\pi}/2\,.
$
\end{abstract}


\maketitle


\section*{Introduction}
The projection constant is a fundamental concept in Banach spaces and their local theory. It has its origins in the study of complemented subspaces of Banach spaces.
If $X$ is
a~complemented subspace of a~Banach space $Y$, then the relative projection constant of $X$ in $Y$ is defined by
\begin{align*}
\boldsymbol{\lambda}(X, Y) & =  \inf\big\{\|P\|: \,\, P\in \mathcal{L}(Y, X),\,\, P|_{X} = Id_X\big\}\\
&  = \inf\big\{c>0: \,\,\text{$\forall\, T \in \mathcal{L}(X, Z)$ \,\, $\exists$\,\, an extension\,
$\widetilde{T}\in \mathcal{L}(Y, Z)$\, with $\|\widetilde{T}\| \leq c\,\|T\|$}\big\}\,,
\end{align*}
where $Id_X$ denotes the identity operator on $X$ and as usual $\mathcal{L}(U,V)$ denotes the Banach space of all bounded linear operators between the Banach spaces $U$ and $V$ with the uniform norm. In what follows $\mathcal{L}(U):=\mathcal{L}(U,U)$. We use here the convention that $\inf \varnothing = \infty.$

The (absolute) projection constant of $X$ is given by
\[
\boldsymbol{\lambda}(X) := \sup \,\,\boldsymbol{\lambda}(I(X),Y)\,,
\]
where the supremum is taken over all Banach spaces $Y$ and isometric embeddings $I\colon X \to Y$.

It is well-known that any Banach space $X$ embeds  isometrically into $\ell_\infty(\Gamma)$, where  $\Gamma$ is a~nonempty set depending on  $X$ (and
$\ell_\infty(\Gamma)$ as usual stands for the Banach space of all bounded scalar-valued functions on $\Gamma$),
and then
\begin{equation}\label{essentialpoint}
  \boldsymbol{\lambda}(X) = \boldsymbol{\lambda}(X, \ell_\infty(\Gamma))\,.
\end{equation}
Thus finding $\boldsymbol{\lambda}(X)$ is equivalent to finding the norm of a~minimal projection from $\ell_\infty(\Gamma)$ onto an isometric copy of
$X$ in $\ell_\infty(\Gamma)$.

Let us recall a few concrete cases relevant for our purposes -
for an extensive treatment on all of this we  refer  to the excellent  monographs
 \cite{ diestel1995absolutely, LT1, pisier1986factorization, tomczak1989banach, wojtaszczyk1996banach}. 
We use standard notation from (local) Banach space theory, and 
 note that, all over the article, all Banach spaces are supposed to be complex. As usual  $\mathcal{L}(X)$  denotes the Banach space of all (bounded) linear
operators $T$ on $X$ together with the operator norm. For  $1\leq p \leq \infty$  
and $n \in \mathbb N$ the symbol $\ell_p^n$ indicates the Banach space  $\mathbb C^n $ 
endowed the  Minkowski norm $\|x\|_p = \big(\sum_{k=1}^n|x_k|^p\big)^{1/p}$
  and $\|x\|_\infty = 
\sup_{1 \leq k \leq n}|x_k|$, respectively.

A well-known simple application of the Hahn-Banach theorem shows that
\[
\boldsymbol{\lambda}(\ell_\infty^n)=1\,.
\]
The exact values of $\boldsymbol{\lambda}(\ell_2^n)$ and $\boldsymbol{\lambda}(\ell_1^n)$ were computed by Gr\"unbaum
\cite{grunbaum} and Rutovitz \cite{rutovitz}: If $d\sigma$ stands for the normalized surface measure on the sphere  $\mathbb{S}_n(\mathbb{C})$, then
\begin{align}
\label{grunbuschC-A}
\boldsymbol{\lambda}\big(\ell_2^n\big)  = n \int_{\mathbb{S}_n} |x_1|\,d\sigma
= \frac{\sqrt{\pi}}{2}   \frac{n!}{\Gamma(n + \frac{1}{2})}\,.
\end{align}
On the other hand, if $dz$ denotes the normalized Lebesgue measure on the distinguished boundary $\mathbb{T}^n$ in $\mathbb{C}^n$
and $J_0$ is the zero Bessel function defined by
$
J_0(t) = \frac{1}{2\pi} \int_{0}^{\infty} \cos( t \cos \varphi) d \varphi\,,
$
then
\begin{align}\label{grunbuschC-B}
\boldsymbol{\lambda}\big(\ell_1^n\big)  = \int_{\mathbb{T}^n} \Big|\sum_{k=1}^{n} z_k\Big|\, dz
= \int_{0}^{\infty} \frac{1 -J_0(t)^n}{t^2} dt\,.
\end{align}
Moreover, K\"onig, Sch\"utt and Tomczak-Jagermann \cite{konig1999projection} proved that for  $1 \le p \le 2$
\begin{align}\label{koenigschuetttomczak}
  \lim_{n\to \infty} \frac{\boldsymbol{\lambda}\big(\ell_p^n\big)}{\sqrt{n}} = \frac{\sqrt{\pi}}{2}\,.
\end{align}

Let us turn to the non-commutative analogs of these results. The operator analog of $\ell_\infty^n$ is the  Banach space
$\mathcal L(\ell_2^n)$.
By \cite[Theorem~5.6]{gordon1974absolutely}  it is known that
\begin{equation*}\label{asy1}
\boldsymbol{\lambda}\big(\mathcal L(\ell_2^n)\big)=\frac{\pi}{4} \frac{n!^2}{\Gamma(n + \frac{1}{2})^2}
\,.
\end{equation*}
The space of Hilbert-Schmidt operators $\mathcal H_2(n)$ on $\ell_2^n$ is a Hilbert space, and we may deduce from~\eqref{grunbuschC-A} that
\[
\boldsymbol{\lambda}\big(\mathcal{H}_2(n)\big) \,\,=\,\, \frac{\sqrt{\pi}}{2}\frac{n^2! }{ \Gamma(n^2 + \frac{1}{2})}\,,
\]
which in particular leads to the two limits
\begin{equation} \label{missing}
\lim_{n \to \infty} \frac{\boldsymbol{\lambda}\big(\mathcal{H}_2(n)\big)}{n} = \frac{\sqrt{\pi}}{2} \,\,\,\,\,\,\,\,
\text{and}\,\,\,\,\,\,\,\, \lim_{n \to \infty} \frac{\boldsymbol{\lambda}\big(\mathcal L(\ell_2^n)\big)}{n} = \frac{\pi}{4}\,.
\end{equation}

Finite dimensional Schatten classes  form  the building blocks
of a variety of natural objects in non-commutative functional analysis. Recall that the singular numbers $(s_k(u))_{k=1}^n$ of  $u \in \mathcal L(\ell_2^n)$ are given by the eigenvalues of $|u| = (u^\ast u)^{1/2}$, and that the Schatten $p$-class
$\mathcal S_p(n), \, 1 \leq p \leq \infty$
by definition is the Banach space  of all operators on $\ell_2^n$ endowed with the norm
$\|u\|_p = \big( \sum_{k=1}^n |s_k(u)|^p  \big)^{1/p}$ (for $p= \infty$ we here  take the maximum over all $1 \leq k \leq n$).
It is well-known that the equalities $\mathcal S_\infty(n) = \mathcal{L}(\ell_2^n)$ and $\mathcal S_2(n) = \mathcal{H}_2(n)$
hold isometrically. We remark that the space $\mathcal S_1(n)$ is usually referred to as trace class.

For the non-commutative analogue of \eqref{grunbuschC-B} in case of  $\mathcal S_1(n)$, the best known estimate
seems
\begin{equation}\label{asytrace}
\frac{n}{3}\le\boldsymbol{\lambda}\big(\mathcal S_1(n)\big)\le n\,.
\end{equation}
The lower bound  was  proved by Gordon and Lewis in \cite{gordon1974absolutely}, while the upper bound is a consequence of the famous Kadets-Snobar inequality \cite{kadecsnobaR}.

As pointed out in  \eqref{grunbuschC-B} there is a useful integral formula for $\boldsymbol{\lambda}\big(\ell_1^n\big)$.
Our main aim is to show a non-commutative analogue for $\boldsymbol{\lambda}\big(\mathcal S_1(n)\big)$, and to employ it to get the missing limit  from \eqref{missing}. More precisely, we prove that
\[
\boldsymbol{\lambda}\big(\mathcal S_1(n)\big) = n \int_{\mathcal U_n} \vert \text{tr}(U) \vert \,dU \,,
\]
where $\text{tr}(U)$ denotes the the trace of the matrix $U$ and the integration is with respect to the Haar
probability measure on the unitary group $\mathcal U_n$;
and then we apply  a probabilistic
approach (within the  so-called Weingarten calculus) to  derive
\[
\dis\lim_{n\to \infty}
\frac{\boldsymbol{\lambda}\big(\mathcal S_1(n)\big)}{n}
= \frac{\sqrt{\pi}}{2}\,.
\]

We finish this introduction with a few words on the technique used.
An important tool to calculate projection constants, and more generally to obtain minimal projections,  is due to Rudin
\cite{rudin1962projections} (see also \cite[Chapter III.B]{wojtaszczyk1996banach}).
This technique is sometimes called Rudin's averaging technique, and it for  example  may be used to prove~\eqref{grunbuschC-A}
as well as~\eqref{grunbuschC-B}.

Given an isometric subspace  $X$  of $Y$, we sketch the most important steps of this strategy
 to find the relative projection constant $\boldsymbol{\lambda}\big(X,Y\big)$: \,
  Select a possible    'natural candidate'  $\mathbf P: Y \to X$ for a~minimal projection;\,
Find a topological group $G$ acting on $\mathcal{L}(Y)$, that is, every $g \in G$ defines operators $T_g$ which
acts in a 'compatible way' on $Y$;\,
Show that $\mathbf P$ in fact is the unique projection, which commutes with all operators $T_g,\, g \in G$;\,
Consider  an arbitrary  projection $\mathbf Q: Y \to X$, and average all operators $T_g^{-1} \mathbf Q T_g$ with respect to the Haar measure
on $G$. Then this average commutes with all operators $T_g,\, g \in G$, and so it  must  coincide with $\mathbf P$;\,
Use a simple convexity argument to show  that $\boldsymbol{\lambda}\big(X, Y\big) = \|\mathbf P: Y \to X\|$;\,
Analyze $\|\mathbf P: Y \to X\|$, in order to  refine the formula for $\boldsymbol{\lambda}\big(X, Y\big)$.

Thus, if here $Y= \ell_\infty(\Gamma)$, then~\eqref{essentialpoint} shows that these steps may lead to a formula/estimate
of $\boldsymbol{\lambda}\big(X)$.
Let us see, how our object  of desire $\mathcal{S}_1(n)$ naturally embeds in some reasonable $\ell_\infty(\Gamma)$.  It is well-known that
$\mathcal L(\ell_2^n)$ and $\mathcal{S}_1(n)$ are in trace duality, that is, the mapping
\begin{equation}\label{dualitytr}
\mathcal{S}_1(n) \to \mathcal L(\ell_2^n)^\ast\,,\,\,\,\,\,\,\,\,\,\,\,\, u \mapsto [v \mapsto \textrm{tr}(uv)]
\end{equation}
defines a~linear and isometric bijection. To go one step further, we may compose this mapping with the restriction
map $\mathcal L(\ell_2^n)^\ast \to C(\mathcal U_n)\,, \,\,u \mapsto u|_{\mathcal U_n}$, where  $\mathcal U_n$ stands for the group of all unitary $n \times n$ matrices, and in fact this leads to an isometric embedding $\mathcal{S}_1(n) \hookrightarrow C(\mathcal U_n)$
(see Proposition~\ref{linco}).
So our aim in the following will be to analyze  the relative projection constant
$\boldsymbol{\lambda}\big(\mathcal{S}_1(n), C(\mathcal U_n)\big)$.

In order to apply  Rudin's avaraging technique  we need to develop what we call
'unitary harmonics' on the unitary group $\mathcal U_n$, which (roughly spoken) are harmonic  polynomials in finitely many 'matrix variables' $z$ and $\overline{z}$ from the  unitary group $\mathcal{U}_n$. All this is deeply inspired by the classical theory of spherical harmonics (see, e.g.,  \cite{rudin1980}), that is, the study of harmonic polynomials in finitely many complex variables $z$ and $\overline{z}$ on
the $n$-dimensional euclidean sphere $\mathbb{S}_n$.
Unitary harmonics and their density in $C(\mathcal{U}_n)$ are described in Section \ref{section: unitary harmonics}.

In
Section~\ref{Projection constants} we
formulate and prove  our  main Theorem~\ref{mainU}.
And although this is the sole focus of this work
- structuring the proof of Theorem~\ref{mainU} carefully, shows that parts of it  extend to a more abstract version
given in Theorem~\ref{abstract}.

\section{Unitary harmonics and their density\label{section: unitary harmonics}}
We need to  extend a few aspects  of the theory of spherical  harmonics on the sphere $\mathbb{S}_n$ (as developed for example in 
\cite[Chapter~12]{rudin1980} and \cite[Chapter~2]{AtkinsonHan2012}) to what we call
unitary harmonics on the unitary 
group~$\mathcal{U}_n$.

\subsection{Unitaries}  \label{Unitaries}
Denote by $M_n$  the space of all $n \times n$-matrices $Z = (z_{k\ell})$ with entries from~$\mathbb{C}$.
The group $\mathcal{U}_n$ of all unitary $n\times n$ matrices $U = (u_{ij})_{1 \leq i,j \leq n}$ endowed with the topology induced by $\mathcal{L}(\ell_2^n)$   forms a~non-abelian compact group. It  is unimodular, and we denote the integral, with respect to the Haar
measure on  $\mathcal{U}_n$, of a function $f \in L_2(\mathcal{U}_n)$  by $$\int_{\mathcal{U}_n} f(U) dU\,.$$ Integrals of this type
form the so-called  Weingarten calculus,  which is  of outstanding importance in random matrix theory, mathematical physics, and the
theory of quantum information (see, e.g., \cite{collins2006integration,kostenberger2021weingarten}).

Basically, we will only need   the precise values  of two concrete integrals from the Weingarten calculus. The first one is
\begin{equation}\label{intformA}
\int_{\mathcal{U}_n} u_{i,j} \overline{u_{k,\ell}}dU = \frac{1}{n} \delta_{i,k}  \delta_{j,\ell}
\end{equation}
for all possible $1\leq i,j,k,\ell \leq n$, and the second one
\begin{equation}\label{intformB}
\int_{\mathcal{U}_n} |\textrm{tr}(AU)|^2 dU = \frac{1}{n} \textrm{tr}(AA^\ast)
\end{equation}
for every $A \in  M_n$ (see, e.g., \cite[p.~16]{cerezo2021cost}, \cite{collins2006integration},
or \cite[Corollary 3.6]{zhang2014matrix}).

Every operator  $T\colon M_n \to M_n$
that leaves $\mathcal U_n$ invariant  (i.e., $T\mathcal U_n\subset \mathcal U_n$), defines a~composition operator
\[
C_{T}\colon L_2(\mathcal U_n) \to L_2(\mathcal U_n)\,,\quad\, f \mapsto f \circ T\,.
\]
There are in fact two such operators $T$, leaving $\mathcal U_n$ invariant,  of special interest -- the left and right multiplication
operators $L_V$ and $R_V$ with respect to $V \in \mathcal U_n $ are given by 
\begin{align*}
L_V (U) := VU \quad \text{and} \quad R_V (U) := UV, \quad\,U \in M_n(\mathbb{C}^n)\,.
\end{align*}
A subspace $S \subset L_2(\mathcal U_n)$ is said to be $\mathcal U_n$-invariant whenever it is invariant under all possible
composition operators $C_{L_V}$ and $C_{R_V}$ with $V \in \mathcal U_n $\,.

For any  closed subspace $S \subset L_2(\mathcal U_n)$, we denote by $\pi_{S}\colon L_2(\mathcal U_n) \to L_2(\mathcal U_n)$
the orthogonal projection on $L_2(\mathcal U_n)$ with range~$S$.

\subsection{Spherical harmonics} \label{real}
The  symbol $\mathcal P (\mathbb{R}^N)$ stands for all  polynomials $f\colon \mathbb{R}^N \to \mathbb{C}$ of the form
 \begin{align} \label{repr}
f\left(x\right)=\sum_{\alpha \in J} c_\alpha\,\,x^\alpha,
\end{align}
where  $J \subset  \mathbb{N}_0^N$ is finite, and $(c_\alpha)_{\alpha\in J}$ are  complex coefficients. Moreover, given  $k \in \mathbb{N}_0$ we write
$\mathcal P_k (\mathbb{R}^N)$  for all $k$-homogeneous polynomials $f$ of this type,
that is, $f$ has a~representation like in \eqref{repr} with  
$|\alpha|:=\sum \alpha_i= k $ for each $\alpha \in J$.

Observe that  in~\eqref{repr} one has $c_\alpha =  \partial^\alpha f(0)/\alpha!$ for each $\alpha \in J$,
which in particular shows the uniqueness of the coefficients
for each $f \in\mathcal P (\mathbb{R}^N)$ (in the following we often write  $ c_\alpha=c_\alpha(f) $).
A particular consequence is that the linear space  $\mathcal P (\mathbb{R}^N)$ carries a natural inner product given by
\begin{equation}\label{inner}
\big\langle f,g\big\rangle_{\mathcal P}:=\sum_{\alpha}\alpha!\,c_\alpha(f)\overline{ c_\alpha(g)}, \quad\,
f,g \in  \mathcal{P}(\mathbb{R}^N)\,.
\end{equation}
This scalar product has a useful reformulation.
To see this, note first that every  $f \in \mathcal{P} (\mathbb{R}^N) $ defines the differential operator
$f(D)\colon \mathcal P (\mathbb{R}^N)\to \mathcal P (\mathbb{R}^N)$ given by
\[
f\left(D\right)g:=\sum_{\alpha}  c_{\alpha}(f) \partial^\alpha g\,,\quad\, g \in \mathcal{P}(\mathbb{R}^N)\,.
\]
And then it is straight forward  to  verify  for every $f,g \in  \mathcal P (\mathbb{R}^N)$ the formula
\begin{equation}\label{diff}
\big\langle f,g\big\rangle_{\mathcal P}=\big[f(D)\overline g\big](0)\,.
\end{equation}
The  polynomial $\mathbf t\in  \mathcal{P}_2 (\mathbb{R}^N)$ defined by
\begin{equation}\label{zahn}
    \mathbf{t}(x):= \|x\|_2^2, \quad\, x\in \mathbb{R}^N
\end{equation}
is of special importance, since then
\[
\Delta : =\mathbf t (D) = \sum_{j=1}^N \frac{\partial^2}{\partial x_{j}^2} \colon
\,\, \mathcal P (\mathbb{R}^N) \to \mathcal P (\mathbb{R}^N)\,
\]
is the Laplace operator. A polynomial $f \in  \mathcal{P}(\mathbb{R}^N)$ is said to be harmonic, whenever
$\Delta f=0$, and we write $\mathcal{H}(\mathbb{R}^N)$ for the subspace of all harmonic polynomials in
$\mathcal{P}(\mathbb{R}^N)$, and $\mathcal{H}_k (\mathbb{R}^N)$ for all $k$-homogeneous, harmonic polynomials. For each $N \in \mathbb{N}$  one has
\begin{equation}\label{ludo1}
    \mathcal H (\mathbb{R}^N)=\text{span}_{k}\mathcal{H}_k (\mathbb{R}^N)\,.
\end{equation}
To see this, fix $f=\sum_{\alpha \in J} c_\alpha(f)\,\,x^\alpha \in \mathcal H (\mathbb{R}^N)$ with degree
$d = \max_{\alpha \in J} |\alpha|$. For each  $k \in \{0, 1,\ldots, d\}$ let
$f_k=\sum_{|\alpha| =k} c_\alpha(f)\,\,x^\alpha$ be the $k$-homogeneous part of $f$. Since $f = \sum_{k=0}^d f_k$,
it remains to show that each $f_k$ is harmonic. Clearly, $\sum_{k=0}^d\triangle f_k =\triangle f=0$. Since  all
$f_k$ are supported on disjoint index sets of multi indices, we conclude that $\triangle f_k =0$ for each
$0 \leq k \leq d$.

Much of what follows is based on the following well-known decomposition of $\mathcal{P}_{k}(\mathbb{R}^N)$ into
harmonic subspaces (see, e.g., \cite[Theorem~2.18]{AtkinsonHan2012}). For the sake of completeness we include
a~proof.

\begin{proposition} \label{realo}
For each $k \in \mathbb{N}_0$ and $N \in \mathbb{N}$
\[
\mathcal{P}_{k}(\mathbb{R}^N)=\mathcal{H}_{k}(\mathbb{R}^N)\,\oplus\,\mathbf t\cdot \mathcal{H}_{k-2}(\mathbb{R}^N)\,\oplus\,\mathbf t^2\cdot  \mathcal{H}_{k-4} (\mathbb{R}^N)\,\oplus\,\,\dots\,\,,
\]
where the orthogonal sum, taken with respect to the inner product from~\eqref{diff}, stops when the subscript reaches $1$ or $0$.
\end{proposition}

\begin{proof}
Given $g\in \mathcal{P}_{k-2} (\mathbb{R}^N)$, we let $h(x):=\mathbf t(x)g(x)$ for all $x \in \mathbb{R}^N$. Since
$\mathbf t(D)=\Delta$, this implies that $h(D)=\Delta \circ g(D)=g(D)\circ \Delta$. Clearly, if now $f\in\mathcal P_k\big(\mathbb{R}^N\big)$,
then by \eqref{diff}
\[
\big\langle h,f\big\rangle_{\mathcal{P}} =\big[h(D)\overline f\big](0) =\big[g(D)\big(\Delta \overline f\big)\big](0)
= \big\langle g,\Delta f\big\rangle_{\mathcal{P}}\,.
\]
\noindent
Thus, $f\perp \mathbf t g$ for every $g\in \mathcal{P}_{k-2}(\mathbb{R}^N)$ is equivalent to  $\Delta f\perp  g$ for every
$g\in  \mathcal{P}_{k-2}(\mathbb{R}^N)$, so to $f \in \mathcal H_k(\mathbb{R}^N)$. As a consequence, we get
\[
\mathcal{P}_k(\mathbb{R}^N)    =\mathcal{H}_k(\mathbb{R}^N)    \oplus \mathbf t\cdot \mathcal{P}_{k-2}(\mathbb{R}^N)\,.
\]
The proof finishes repeating this procedure for $\mathcal{H}_{k-2}(\mathbb{R}^N), \,\mathcal{H}_{k-4}(\mathbb{R}^N), \,\,\ldots$
\end{proof}

By $\mathbb S_N^{\mathbb{R}}$ we denote the sphere in the real Hilbert space $\ell_2^N(\mathbb{R})$. We write
$\mathcal{P}(\mathbb S_N^{\mathbb{R}})$ for the linear space of all restrictions  $f|_{\mathbb S_N^{\mathbb{R}}}$ of polynomials
$f \in \mathcal{P}(\mathbb{R}^N)$, and $\mathcal{P}_k(\mathbb S_N^{\mathbb{R}})$, whenever we only consider restrictions of
$k$-homogeneous polynomials.

All restrictions of harmonic polynomials on $\mathbb{R}^N$, so polynomials in $\mathcal{H}(\mathbb{\mathbb{R}^N})$,
to the sphere $\mathbb S_N^{\mathbb{R}}$  are denoted~by
$
\mathcal{H}(\mathbb S_N^{\mathbb{R}})\,,
$
and  such  polynomials are called spherical harmonics. Similarly, we denote by  $\mathcal{H}_k(\mathbb S_N^{\mathbb{R}})$ the
space collecting all  $k$-homogeneous polynomials restricted to $\mathbb S_N^{\mathbb{R}}$. Endowed with the supremum norm taken
on $\mathbb S_N^{\mathbb{R}}$, both spaces $\mathcal{H}(\mathbb S_N^{\mathbb{R}})$ and $\mathcal{H}_k(\mathbb S_N^{\mathbb{R}})$
form subspaces of $C(\mathbb S_N^{\mathbb{R}})$.

An important fact, not needed here, is that the spaces $\mathcal{H}_k(\mathbb S_N^{\mathbb{R}})$ are pairwise orthogonal
in~$L_2(\mathbb S_N^{\mathbb{R}})$ 
(see, e.g., \cite[Corollary~2.15]{AtkinsonHan2012}).

\subsection{Unitary harmonics} \label{matrix}
Going one step further, we extend the notion of spherical harmonics on the real sphere $\mathbb{S}_N^{\mathbb{R}}$ to what
we  call unitary harmonics on the unitary group $\mathcal{U}_n$.

Recall that  $M_n$  here stands for the space of all $n \times n$-matrices $Z = (z_{k\ell})$ with entries from~$\mathbb{C}$.
 The subset of
such matrices $\alpha  = (\alpha_{k\ell})$ with entries from  $\mathbb{N}_0$ is denoted by $M_n(\mathbb{N}_0)$. For
$Z \in M_n$ and $\alpha  = (\alpha_{k\ell}) \in  M_n(\mathbb{N}_0)$  we define
\[
Z^\alpha = \prod_{k,\ell=1}^{n} z_{k\ell}^{\alpha_{k\ell}}\,.
\]
We  identify $M_n$ with $\mathbb{R}^{2n^2}$ in the canonical way through the  bijective mapping
\begin{equation}\label{idi}
\mathbf{I}_n \colon M_n \longrightarrow \mathbb{R}^{2n^2}\,,
\end{equation}
which assigns to every matrix $Z = (z_{k\ell})_{k\ell}=(x_{k\ell} + i y_{k\ell})_{k\ell}\in M_n$ the element
\[
\big(x_{11},\, y_{11}, \ldots, x_{1n},\, y_{1n},\, x_{21}, \, y_{21},\ldots,x_{2n}, \, y_{2n},
\cdots, x_{n1},\, y_{n1}, \ldots, x_{nn}, \, y_{nn} \big) \in \mathbb{R}^{2n^2}\,.
\]

Then  $\mathfrak P (M_n)$ denotes   the linear space of all polynomials $f = g \circ \mathbf{I}_n$
with $g \in \mathcal{P}\big(\mathbb{R}^{2n^2}\big)$. Hence, by definition, the mapping
\begin{equation}\label{equo-dos}
\mathcal{P}\big(\mathbb{R}^{2n^2}\big) \,\,= \,\,\mathfrak{P}(M_n)\,, \quad g \mapsto g\circ \mathbf{I}_n
\end{equation}
identifies both  spaces as vector spaces.

We collect a couple of useful facts.
Note first that, if  $f=g \circ \mathbf{I}_n \in \mathfrak{P} (M_n)$ with $g \in \mathcal{P} (\mathbb{R}^{2n^2})$, then
\begin{equation*} \label{jo1}
  \Delta f = \sum_{i,j=1}^n\frac{\partial^2f}{\partial z_{ij}\partial\overline{z}_{ij}}
 \,= \,\frac{1}{4}\,\sum_{i,j=1}^n  \Big(\frac{\partial^2 g}{\partial x_{ij}^2} +\frac{\partial^2 g}{\partial y_{ij}^2}\Big)\,,
\end{equation*}
a formula immediate from the definition of
 $\partial_{z_{ij}} = \frac{1}{2}(\partial_{x_{ij}}-i\partial_{y_{ij}})$ and 
$\partial_{\overline{z}_{ij}}=  \frac{1}{2}(\partial_{x_{ij}}+i\partial_{y_{ij}})$. 

Secondly,  a~function $f\colon M_n \to \mathbb{C}$ belongs to $\mathfrak{P}(M_n)$ if and only
if it has a~representation
\begin{equation} \label{representationA}
    f(Z)=\sum_{(\alpha,\beta) \in J} c_{(\alpha,\beta)}\,\,Z^\alpha\overline{Z}^\beta\,,\quad \,\,Z \in M_n\,,
\end{equation}
where $J$ is a finite index set in $M_n(\mathbb{N}_0) \times M_n(\mathbb{N}_0)$ and $c_{(\alpha,\beta)} \in \mathbb{C},\,(\alpha,\beta) \in J$.
Moreover, in this case this representation is unique.

Indeed, if $f$ is given by~\eqref{representationA}, then $g = f \circ \mathbf{I}_n^{-1} \in \mathcal{P}\big(\mathbb{R}^{2n^2}\big)$ and $f = g \circ \mathbf{I}_n \in \mathfrak{P}(M_n)$. Conversely, if $f=g \circ \mathbf{I}_n \in \mathfrak{P}(M_n)$ with $g=\sum_{\alpha} c_\alpha\,x^\alpha \in \mathcal{P}(\mathbb{R}^{2n^2})$, then the desired representation easily follows from the substitution: $\re z_{ij} =\frac{1}{2}( z_{ij}+\overline{z}_{ij})$
and $\text{Im}\, z_{ij} =\frac{1}{2}( z_{ij}-\overline{z}_{ij})$. 
To see the uniqueness of the representation  in \eqref{representationA}  observe that  if  $f=0$,  then, 
given  $(\alpha, \beta) \neq (0,0)$, an application of  the differential operator
\[
\partial_{z_{11}}^{\alpha_{11}}\ldots \partial_{z_{1n}}^{\alpha_{1n}}\,
\partial_{\overline{z}_{11}}^{\beta_{11}}\ldots \partial_{\overline{z}_{1n}}^{\beta_{1n}}
\dots  \dots
\partial_{z_{n1}}^{\alpha_{n1}}\ldots \partial_{z_{n1}}^{\alpha_{nn}}\,
\partial_{\overline{z}_{n1}}^{\beta_{n1}}\ldots \partial_{\overline{z}_{nn}}^{\beta_{nn}}
\]
to $f$ (and evaluating in zero),  shows that  $c_{(\alpha,\beta)} = 0$.

We again use the identification $g \mapsto g\circ \mathbf{I}_n$ from  \eqref{idi} to define the spaces
\begin{equation}\label{okko}
  \text{
  $  \mathfrak{P}_k(M_n):=\mathcal{P}_k\big(\mathbb{R}^{2n^2}\big)\,,
  \quad
    \mathfrak{H}_k(M_n):=\mathcal{H}_k\big(\mathbb{R}^{2n^2}\big)$
  \quad and \quad
  $  \mathfrak{H}(M_n):=\mathcal{H}\big(\mathbb{R}^{2n^2}\big)$\,.}
\end{equation}

The following  lemma gives a~simple description of the elements of  $\mathfrak{P}_k(M_n)$.

\begin{lemma} \label{ludo1000}
Let $ f \in \mathfrak P (M_n)$ and $k \in \mathbb{N}$. Then $ f \in \mathfrak P_k (M_n)$  if and only if
$f(\lambda Z) = \lambda^kf(Z)$ for all $\lambda \in \mathbb{R}$ and $Z \in M_n$.
\end{lemma}

\begin{proof}
For $f \in \mathfrak P_k (M_n)$ there is $g \in \mathcal{P}_k\big(\mathbb{R}^{2n^2}\big)$ such that $f= g\circ \mathbf{I}_n$.
Clearly, $f(\lambda Z) = \lambda^kf( Z)$ for  all  $\lambda \in \mathbb{R}$ and $Z\in M_n$. Assume conversely that $f$ is
$k$-homogeneous in the meaning of the statement. Since $ f \in \mathfrak{P}(M_n)$, there is a~finite polynomial
$g\left(x\right)=\sum_J c_\alpha(g)\,x^\alpha$,\,$x \in \mathbb{R}^{2n^2}$ such that $f = g \circ \mathbf{I}_n$. Since
$g = f \circ \mathbf{I}_n^{-1}$, it follows that $g(\lambda x) = \lambda^kg( x)$ for all  $\lambda \in \mathbb{R}$ and
$x \in \mathbb{R}^{2n^2}$. But this  by the uniqueness of the coefficients $c_\alpha(g)$ necessarily implies that
$c_\alpha(g) \neq 0$ only if $|\alpha|=k$, so as desired $g \in \mathcal{P}_k(\mathbb{R}^{2n^2})$.
\end{proof}

Obviously,
\begin{equation}\label{denso1}
\mathfrak{P}(M_n)=\text{span}_k \mathfrak P_k ( M_n) 
\end{equation}
(consider the polynomials on $\mathbb{R}^{2n^2}$ defining these spaces), and less trivially (as an immediate consequences of
\eqref{ludo1}) we have
\begin{equation}\label{denso5}
\mathfrak H (M_n)=\text{span}_k \mathfrak H_k (M_n)\,.
\end{equation}
The  polynomial $\mathbf{t}_{M_n}\in \mathfrak P_2\big(M_n\big)$ given by
\[
\mathbf{t}_{M_n}(Z):=\mathrm{tr}(ZZ^*), \quad\, Z\in M_n\,,
\]
where $\mathrm{tr}\colon M_n \to \mathbb{C}$ denotes the trace, is of special importance.
It is easily seen that under the  identification from  \eqref{equo-dos} the image of  the polynomial
$\mathbf{t} \in \mathcal{P}_2\big(\mathbb{R}^{2n^2}\big)$ (see again \eqref{zahn}) is $\mathbf{t}_{M_n}\in \mathfrak P_2\big(M_n\big)$,
that~is,
\begin{equation} \label{t=t}
\mathbf{t}_{M_n}(Z) = \mathbf{t}\big(\mathbf{I}_{n} Z\big), \quad\, Z \in M_n\,.
\end{equation}
Recall again that $\mathcal{P}(\mathbb{R}^{2n^2})$ carries the  natural inner product from~\eqref{inner},
which then by the identification in~\eqref{equo-dos} transfers  to a natural inner product on  $\mathfrak{P}(M_n)$, that is,
\begin{equation}\label{innerB}
\big\langle f, g\big\rangle_{\mathcal P}:= \big\langle f \circ \mathbf{I}_n^{-1}, g \circ \mathbf{I}_n^{-1}\big\rangle_{\mathcal P},
\quad\, f,\, g \in \mathfrak{P}(M_n)\,.
\end{equation}
Using~\eqref{okko} and~\eqref{t=t}, we deduce from Proposition~\ref{realo} its matrix analog,
which is going to be of great value later on.

\begin{proposition}\label{matrixo}
For all  $k \in \mathbb{N}_0$ and $n \in \mathbb{N}$
\begin{equation*}
  \mathfrak P_k (M_n)=\mathfrak H_k (M_n)\,\oplus\,\mathbf{t}_{M_n}\cdot \mathfrak H_{k-2} (M_n)\,\oplus\,
  \mathbf{t}_{M_n}^2\cdot \mathfrak H_{k-4} (M_n)\,\oplus\,\,\dots\,\,,
\end{equation*}
where the last term of the orthogonal sum is the span of $\mathbf{t}_{M_n}^{k/2}$ for even $k$, and
$\mathbf{t}_{M_n}^{(k-1)/2}\cdot \mathfrak H_1(M_n)   $ for odd~$k${\rm}.
\end{proposition}

We need two more lemmas.

\begin{lemma} \label{ludo100}
Let $f \in  \mathfrak{H}(M_n)$ and $U \in \mathcal{U}_n$. Then $f \circ L_U \in  \mathfrak{H}(M_n)$.
Moreover, if $f \in \mathfrak{H}_k(M_n)$, then also $f \circ L_U \in  \mathfrak{H}_k(M_n)$.
\end{lemma}
\begin{proof}
Recall  the well-known fact
that for every harmonic function $F\colon \mathbb{C}^{n^2} \to \mathbb{C}$ and every $W \in \mathcal{U}_{n^2}$,
we have that $\triangle (F \circ \Phi_W ) = \triangle F \circ \Phi_W$,  where $\Phi_W z = Wz$  for
$z \in \mathbb{C}^{n^2}$. Now identify $M_n$ and $\mathbb{C}^{n^2}$ in the natural way by
\begin{equation}\label{idddo}
\mathbf{J}_n(Z) = \big(z_{11}, \ldots, z_{1n}, \, z_{21} \ldots, z_{2n}, \cdots, z_{n1}, \ldots, z_{nn}\big), \quad Z \in M_n\,,
\end{equation}
and define $g = f \circ \mathbf{J}^{-1}_n: \mathbb{C}^{n^2} \to \mathbb{C} $. Then obviously $\triangle g = 0$, and  moreover
a~simple calculation shows
\[
f \circ L_U = g \circ \Phi_{U \otimes \text{id}_{\mathbb{C}^n}} \circ \mathbf{J}_{n}\,.
\]
Since $U \otimes \text{id}_{\mathbb{C}^n} \in \mathcal{U}_{n^2}$ is unitary, it follows that
\begin{equation*}
\triangle(f \circ L_U) =  \triangle(g \circ \Phi_{U \otimes \text{id}_{\mathbb{C}^n}})
= \triangle g \circ \Phi_{U \otimes \text{id}_{\mathbb{C}^n}} =0\,.
\end{equation*}
For the second statement note that if $f \in \mathfrak{H}_k(M_n)$, then by the first statement $f \circ L_U \in \mathfrak{H}(M_n)$.
But $f \circ L_U (\lambda Z) = \lambda f \circ L_U (Z) $  for all $Z \in M_n$ and $\lambda \in \mathbb{R}$, and hence the claim
follows from Lemma~\ref{ludo1000}. 
\end{proof}

For $p,q\in\mathbb N_0$ let $\mathfrak H_{(p,q)}(M_n)\subset \mathfrak{H}(M_n)$ be the subspace of all harmonic polynomials,
which are $p$-homogeneous in $Z=(z_{ij})$ and $q$-homogeneous in $\overline{Z}=(\overline{z}_{ij})$, that is, all polynomials
$f \in \mathfrak{H}(M_n)$ of the form
\begin{align} \label{representation}
f\left(Z\right)=\sum_{
|\alpha|=p, |\beta|=q} c_{(\alpha,\beta)}\,\,Z^\alpha\overline{Z}^\beta\,, \quad Z \in M_{n}.
\end{align}
By Lemma~\ref{ludo1000} we immediately see that
\begin{equation}\label{mittag}
\mathfrak H_{(p,q)}(M_n)\subset \mathfrak{H}_{p+q}(M_n)\,.
\end{equation}

The following result is crucial for our purpose
(see also Lemma~\ref{ludo200}).

\begin{lemma} \label{ludo13}
For all  $f \in \mathfrak H_{(p,q)}(M_n)$ and $U \in \mathcal{U}_n$ one has
\begin{equation*} \label{AAA}
f \circ L_U \in \mathfrak H_{(p,q)}(M_n)\qquad\text{ and }\qquad f\circ R_U  \in \mathfrak H_{(p,q)}(M_n)\,.
\end{equation*}
Moreover,
  \begin{equation}\label{denso4}
  \mathfrak H (M_n)=\text{span}_{p,q}\mathfrak H_{(p,q)}(M_n)\,.
\end{equation}
\end{lemma}

\begin{proof}
Taking for $f$ a~representation as in \eqref{representation}, we have
\[
f \circ L_U (Z) =\sum_{|\alpha|=p, |\beta|=q}c_{(\alpha,\beta)}\,\,(UZ)^\alpha (\overline{UZ})^\beta\,, \quad Z \in M_n.
\]
Now for each $1 \leq i,j \leq n$ we use the multinomial formula for $(\sum_{\ell} u_{i\ell}z_{\ell j}) ^{\alpha_{i j}}$,
to get
\[
(UZ)^\alpha (\overline{UZ})^\beta = \sum_{|\gamma|\leq p, |\zeta|\leq q }
d_{(\gamma,\delta)}\,\, Z^\gamma \overline{Z}^\zeta\,, \quad Z = (z_{k\ell})_{k,\ell}\in M_n\,.
\]
Combining, we conclude that $f \circ L_U$ has a~representation
\[
f \circ L_U (Z) =\sum_{|\eta| \leq p, |\sigma|\leq q }
e_{(\eta,\sigma)}\,\,Z^\eta \overline{Z}^\sigma\,, \quad Z \in M_n.
\]

On the other hand, 
by ~\eqref{mittag} and Lemma~\ref{ludo100}, it follows that $f \circ L_U \in \mathfrak{H}_{p+q}(M_n)$, and hence 
for all $\lambda \in \mathbb{R}$ and $Z \in M_n$ one has
\begin{align*}
\sum_{|\eta| \leq p, |\sigma|\leq q } e_{(\eta,\sigma)}\,\lambda^{|\eta|+|\sigma|} \,Z^\eta \overline{Z}^\sigma
= (f \circ L_U)(\lambda Z) = \lambda^{p+q}(f \circ L_U)( Z) = \sum_{|\eta| \leq p, |\sigma|\leq q}
c_{(\eta,\sigma)}\,\lambda^{p+q} \,Z^\eta \overline{Z}^\sigma\,.
\end{align*}
Inserting 
$Z = \id \in M_n$,
shows that $e_{(\eta,\sigma)} \neq 0$  only if  $|\eta| + |\sigma|= p+q$, and since 
$|\eta| \leq p$ and $|\sigma| \leq q$, this is only possible whenever $|\eta| = p$ and $|\sigma| = q$. This as desired proves $f \circ L_U \in \mathfrak H_{(p,q)}(M_n)$.

The equality    \eqref{denso4} 
 follows from \eqref{denso5},
since it may  easily be seen that 
$\mathfrak H_k (M_n)=\text{span}_{p+q=k}\mathfrak H_{(p,q)}(M_n)$
(see also, e.g., \cite[Proposition 12.2.2]{rudin1980}).

In order to prove  that $f \circ R_U \in  \mathfrak H_{(p,q)} (M_n)$, define $f^\ast(Z)= f (Z^\ast)$ for $Z \in M_n$.
Since the mapping $\ell_{2}^{n^2} \to \ell_{2}^{n^2},  Z \mapsto Z^\ast$ is unitary (it is an isometry), the function  $f^\ast$ is harmonic. Now,  looking at the representation of $f$ as in \eqref{representation}, we see  that $f^\ast \in \mathfrak{H}_{(q,p)} (M_n)$. This, by what is already proved, gives that  $f^\ast \circ L_{U^\ast}\in \mathfrak{H}_{(q,p)} (M_n)$.
But  for $Z \in M_n$
\begin{align*}
  f \circ R_U (Z) = f(ZU)
  = f((U^\ast Z^\ast)^\ast) = f^\ast(U^\ast Z^\ast)= f^\ast \circ L_{U^\ast} (Z^\ast) =
  (f^\ast \circ L_{U^\ast})^\ast(Z)\,,
\end{align*}
and hence $f \circ R_U  = (f^\ast \circ L_{U^\ast})^\ast \in \mathfrak H_{(p,q)} (M_n)$.
\end{proof}

\subsection{Unitarily invariant subspaces of $C(\mathcal{U}_n)$}
\label{Unitarily invariant subspaces}

By $\mathfrak P (\mathcal U_n)$ and $\mathfrak P_k (\mathcal U_n)$ we denote the linear space of all restrictions
$f|_{\mathcal{U}_n}: \mathcal U_n \to \mathbb{C}$ of polynomials
$f \in \mathfrak P( M_n)$ and $f \in \mathfrak P_k( M_n)$, respectively.

Similarly, we for all restrictions to $\mathcal U_n$ of harmonic polynomials from $\mathfrak H (M_n)$  and $\mathfrak H_k (M_n)$ write
$
\mathfrak H (\mathcal U_n)
$
and
$
\mathfrak H_k (\mathcal U_n)\,,
$
respectively, and the elements therein we address as  unitary harmonics.
All this  constitutes
important subspaces of $C(\mathcal{U}_n)$.

\begin{lemma} \label{ludo300}
For each $k$
\begin{equation}\label{forgot}
\mathfrak P_k (\mathcal U_n)=\text{span}_{\ell\leq k} \,\mathfrak{H}_\ell (\mathcal U_n)\,,
\end{equation}
and 
\begin{align} \label{denso1000}
    \mathfrak P (\mathcal U_n)=\text{span}_k\mathfrak{P}_k (\mathcal U_n)
    =
    \text{span}_\ell\mathfrak{H}_\ell (\mathcal U_n) = \mathfrak H (\mathcal U_n)\,.
    \end{align}
\end{lemma}

\begin{proof}
Proposition~\ref{matrixo} 
and the fact that  the function $\mathbf 
t_{M_n} =n$ on $\mathcal U_n$ imply
~\eqref{forgot}. To prove   
~\eqref{denso1000}
note that  the first equality is a consequence of ~\eqref{denso1}, the second one of~\eqref{forgot}, and the last one 
    of~\eqref{denso5}. 
\end{proof}

Moreover, for $p,q\in\mathbb N_0$, we write $\mathfrak H_{(p,q)}(\mathcal U_n)$
for all  restrictions to $\mathcal U_n$ of functions in $\mathfrak H_{(p,q)}(M_n)$. Observe that  a 
function 
$f:\mathcal U_n  \to \mathbb C$
belongs to $\mathfrak H_{(p,q)}(\mathcal U_n)$
if and only if it has on $\mathcal U_n$ a representation like in~
\eqref{representation}.
All needed information on these subspaces 
of $C(\mathcal{U}_n)$ is included in the following lemma, which is an immediate consequence of  Lemma~\ref{ludo13}.    
\begin{lemma} \label{ludo200}
Each  space  $\mathfrak H_{(p,q)} (\mathcal U_n)$
is a $\mathcal U_n$-invariant subspace of $C(\mathcal U_n)$, that is,
for all $f \in \mathfrak H_{(p,q)} (\mathcal U_n)$
and $U \in \mathcal{U}_n$, we have  $f \circ L_U, f \circ R_U  \in \mathfrak H_{(p,q)} (\mathcal U_n)$. Moreover,
\begin{equation}\label{denso50}
  \mathfrak H (\mathcal U_n)=\text{span}_{p,q}\mathfrak H_{(p,q)}(\mathcal U_n) \,.
  \end{equation}
\end{lemma}

We finish with  the following density result being  crucial for our purposes.

\begin{theorem} \label{den-matrix}
$\mathfrak H (\mathcal U_n)$ is dense in $C(\mathcal U_n)$. In particular, the  span of the union of 
all $\mathfrak H_{k}(\mathcal U_n)$ as well as the span of the union of all $\mathfrak H_{(p,q)}(\mathcal U_n)$ are dense  in $C(\mathcal U_n)$.
\end{theorem}

\begin{proof}
Observe first that $\mathfrak{P} (\mathcal U_n)$ is a~subalgebra of $C(\mathcal U_n)$, which is  closed under conjugation, 
and that the collection of all coordinate functions $e_{ij}$ separates the points of $\mathcal U_n$. Thus, by the 
Stone-Weierstrass theorem, $\mathfrak{P}(\mathcal U_n)$ is dense in $C (\mathcal U_n)$. The rest follows from~\eqref{denso1000}
and~\eqref{denso50}.
\end{proof}

\begin{remark}
An important difference between spherical harmonics and unitary harmonics  is that for the case of the sphere the corresponding 
spaces $\mathfrak{H}_{(p,q)}(\mathbb{S}^{\mathbb{C}}_n)$ are mutually orthogonal in $L_2(\mathbb{S}^{\mathbb{C}}_n)$ 
(see \cite[Theorem 12.2.3]{rudin1980}). But for the subspaces $\mathfrak{H}_{(p,q)}(\mathcal{U}_n)$ of $L^2(\mathcal{U}_n)$ this 
is no longer true. To see an example take $f\in \mathfrak{H}_{(1,0)}(\mathcal{U}_n)$ and
$g\in \mathfrak{H}_{(2,1)}(\mathcal{U}_n)$ defined by
$f(U)=u_{1,1}$ and $g(U)=\overline{u_{2,2}}u_{1,2}u_{2,1}$. Then (see, e.g., \cite[Section 4.2]{hiai2000semicircle})
\begin{equation}\label{no-ortho}
  \big\langle f,g\big\rangle_{L_2}=\int_{\mathcal U_n}u_{1,1}u_{2,2}\overline{u_{1,2}u_{2,1}}dU=-\frac1{(n-1)n(n+1)}\,.
\end{equation}
\end{remark}

On the other hand, using basic properties of the Haar measure  on $\mathcal U_n$, it is not difficult to prove that

\begin{equation} \label{otto}
\text{ $\mathfrak{H}_{(p,q)}(\mathcal{U}_n)\perp \mathfrak{H}_{(p',q')}(\mathcal{U}_n)$ \,\,\,\,
whenever \,\,\,\, $p+q=p'+q'$ \,\,\,\, and \,\,\,\, $(p,q) \neq (p',q')$}
\end{equation}
(see \cite[\S 29]{hewitt2013abstract}, or \cite{kostenberger2021weingarten}).

It is worth noting the following conclusion from \eqref{otto} (not needed for our further purposes), which states that
\[
\mathfrak H_k(\mathcal{U}_n)= \mathfrak H_{(k,0)}(\mathcal{U}_n)\, \oplus\, \mathfrak H_{(k-1,1)}(\mathcal{U}_n)
\oplus\dots \, \oplus\, \mathfrak H_{(0,k)}(\mathcal{U}_n),\, 
\]
where $\oplus$ indicates the orthogonal sum in $L_2(\mathcal U_n)$.
We conclude with the observation  that in contrast to \eqref{no-ortho} we have $\big\langle f,g\big\rangle_{\mathfrak P}= 0$, 
so the Euclidean structure, which $\mathfrak{H}_{(p,q)}(\mathcal{U}_n)$ inherits from $L_2(\mathcal{U}_n)$, is different 
from that induced by the inner product from \eqref{innerB}.

\section{Projection constants} \label{Projection constants}
As explained in the introduction the main goal of this work is to prove the following  result.

\begin{theorem}
\label{mainU}
For each $n \in \mathbb{N}$,
\begin{equation}\label{intgralformula}
   \boldsymbol{\lambda}\big(\mathcal S_1(n)\big) = \big\|\pi_{(1,0)}:C(\mathcal U_n) \to \mathcal S_1(n)\big\| = n \int_{\mathcal{U}_n }|\mathrm{tr} (V)| d V\,.
   \end{equation}
Moreover,
\begin{equation}\label{cor: lim proj constant trace}
  \lim_{n\to \infty} \, \frac{\boldsymbol{\lambda}\big(\mathcal S_1(n)\big)}{n}=\frac{\sqrt \pi}{2}\,.
\end{equation}
\end{theorem}

The proof of this theorem is presented in Section~\ref{proof of the main result}. It is based on preliminary results we prove in the following, which require some
preliminary arguments.

\subsection{Rudin's averaging technique}
Given a topological group $G$   and a Banach space $Y$, we say that $G$    acts   on  $Y$ (through $T$) whenever there is a mapping
\[
T\colon G \to \mathcal{L}(Y)\,, \, \,\, \,g \mapsto T_g
\]
such that
\[
T_e = I_Y, \quad\, T_{gh} = T_g T_h, \quad\,  g, \, \, h\in G
\]
and all mappings
\begin{equation*}\label{con(i)}
 g\ni G \mapsto T_g(y) \in Y,  \quad\,   \, \, y\in Y
\end{equation*}
are continuous.  If in addition all operators $T_g, g\in G$ are isometries, then we say that $G$ acts isometrically on $Y$. We say that $S \in \mathcal{L}(Y)$ commutes with the action $T$ of $G$ on $Y$  whenever $S$ commutes with  all $T_h,\, h\in G$.

The following theorem was presented in \cite{rudin1962projections}  (see also \cite[Theorem III.B.13]{wojtaszczyk1996banach}).
\begin{theorem} \label{rudy}
Let $Y$ be a~Banach space, $X$ a~complemented subspace of\, $Y$, and \,$\mathbf{Q} \colon Y \to Y$ a~projection onto $X$. Suppose that  $G$ is a compact group with Haar measure $\mathrm{m}$, which   acts on $Y$ through $T$ such that $X$ is invariant under the action of $G$, that is,
$T_g(X) \subset X$ for all $g\in G$.
Then $\mathbf{P}\colon Y \to Y$ given by
\begin{equation}\label{equation rudy}
\mathbf{P}(y):= \int_{G} T_{g^{-1}}\mathbf{Q}T_g(y)\,d\mathrm{m}(g), \quad\, y\in Y\,,
\end{equation}
is a~projection onto $X$ which commutes with the action  of $G$ on $Y$ $($meaning that  $T_g\mathbf{P} = \mathbf{P}T_g$ for all
$g\in G$$)$ and satisfies
\[
\|\mathbf{P}\| \,\leq\, \|\mathbf{Q}\|\,\,\,\sup_{g\in G}\|T_g\|^2\,.
\]
Moreover, if there is a unique projection on $Y$ onto $X$ that commutes with the action of $G$ on $Y$, and if $G$ acts isometrically  on $Y$, then $\mathbf{P}$ given in \eqref{equation rudy} is minimal, i.e.,
\[
\boldsymbol{\lambda}(X,Y) = \|\mathbf{P}\|\,.
\]
\end{theorem}

In order to be able to apply  Rudin's technique, we need to  endow  $\mathcal U_n \times \mathcal U_n$
with a special group structure, which allows to represent the resulting group in   $\mathcal{L}(C(\mathcal{U}_n))$.
To do so, consider on $\mathcal U_n \times \mathcal U_n$  the  multiplication
\[
(U_0,V_0)\cdot (U_1,V_1):=(U_1U_0,V_0V_1)\,.
\]
With this multiplication and endowed with the product topology, $\mathcal U_n \times \mathcal U_n$ turns into a compact topological group,
and it may be seen easily  that the Haar measure on $\mathcal U_n \times \mathcal U_n$ is given by the product measure of the Haar measure on $\mathcal U_n $
with itself.

Further, for  any $(U,V) \in \mathcal U_n \times \mathcal U_n $ and  any $f \in L_2(\mathcal U_n)$ we define
\begin{align*}\label{action of U^2}
    \rho_{(U,V)}f : = (C_{L_U}\circ C_{R_V})f = f\circ L_U\circ R_V\,,
\end{align*}
which leads to an action of $\mathcal U_n \times \mathcal U_n$ on $C(\mathcal{U}_n)$
given by
\begin{equation} \label{actU}
  \mathcal U_n \times \mathcal U_n  \to \mathcal{L}\big(C(\mathcal{U}_n)\big)\,,\,\,\, \,\,\,
  (U,V) \mapsto \big[\rho_{(U,V)}: f \mapsto f\circ L_U\circ R_V\big] \,.
\end{equation}
We say that a mapping  $T: S_1 \to S_2 $, where
$S_1$ and $S_2$ both  are $\mathcal{U}_n$-invariant subspaces of $L_2(\mathcal{U}_n)$, commutes with the action of
$\mathcal{U}_n \times \mathcal{U}_n$
on $C(\mathcal{U}_n)$, whenever
\[
\text{$\big(C_{L_U}\circ C_{R_{V}}\big)(Tf)=T\big(( C_{L_U}\circ C_{R_V})f\big)$
\quad for every \quad $(U,V)\in \mathcal U_n \times \mathcal U_n$ and $f\in S_1$\,.}
\]

\subsection{Convolution} Recall from Section~\ref{Unitaries}
that $\pi_S: L_2(\mathcal U_n) \to S$ denotes
the orthogonal projection on $L_2(\mathcal U_n)$ onto a given closed subspace $S$. The following result
shows that under the assumption of $\mathcal U_n$-invariance of $S$, this projection is a convolution operator with respect to some kernel in $S$.

\begin{theorem} \label{newperspective}
Let $S$ be a $\mathcal U_n$-invariant subspace of  $C(\mathcal{U}_n)$, which is closed in  $L_2(\mathcal U_n)$.
Then the following holds true{\rm:}
\begin{itemize}
\item[(i)]There is a unique function $\mathrm{t}_S\in S$ such that for all  $f\in L_2(\mathcal{U}_n)$
\[
\pi_S f = f \ast \mathrm{t}_S\,,
\]
\item[(ii)]
$\pi_S$ commutes with all $L_U$ and $R_U$ for $U\in\mathcal{U}_n$, that is, $\pi_S$ commutes with the action of
$\mathcal{U}_n \times  \mathcal{U}_n$
on $C(\mathcal{U}_n)$\,,
\vspace{2mm}
\item[(iii)]
$\|\pi_S: C(\mathcal{U}_n) \to S\|  = \int_{\mathcal{U}_n }|\mathrm{t}_S (V)| dV $\,.
\end{itemize}

\end{theorem}

The proof is an easy consequence of the following lemma.

\begin{lemma}\label{theorem kernel 12.2.5}
Let $S$ be a $\mathcal U_n$-invariant subspace of  $C(\mathcal{U}_n)$, which is closed in  $L_2(\mathcal U_n)$.
 Then for every  $U\in\mathcal{U}_n$, there exists a unique function $K_U^S\in S$ such that
for all $f\in L_2(\mathcal{U}_n)$
\begin{itemize}
\item[(i)]\label{12.2.5 (1,3)}  $(\pi_Sf)(U)=\big\langle f, K_U^S\big\rangle_{L_2}
=\int_{\mathcal{U}_n}f(V)\overline{K_U^S(V)}dV\,,$
\end{itemize}
and moreover for every choice of $U,V\in\mathcal U_n$ we have
\begin{itemize}
\item[(ii)]
\label{12.2.5 (2)} $K_U^S(V)= \big\langle K_U^S, K_V^S\big\rangle_{L_2}=\overline{K_V^S(U)}$\,,
\vspace{1mm}
\item[{(iii)}]\label{12.2.5 (4)}
$
K_U^S\circ L_{V^{-1}}=K_{VU}^S=K_V^S\circ R_{U^{-1}}\,,
$
\vspace{1mm}
\item[(iv)]\label{12.2.5 (6)} $K_V^S(V)=K_{Id}^S(Id)>0\,.$
\end{itemize}
\end{lemma}

\begin{proof}
The claim from $(i)$ is an immediate consequence of the Riesz representation theorem applied to the continuous linear functional
$
L_2(\mathcal U_n)  \to  \mathbb{C}\,,\,\,\, f \mapsto (\pi_Sf)(U)\,.
$

\noindent $(ii)$ $K_U^S(V)=\pi_S(K_U^S)(V)=\big\langle K_U^S, K_V^S\big\rangle_{L_2}=\overline{\big\langle K_V^S, K_U^S\big\rangle_{L_2}}=\overline{K_V^S(U)}$
for all $V\in\mathcal U_n$.

\noindent $(iii)$
Fix some $ V\in \mathcal U_n$ and
$f\in L_2(\mathcal U_n)$, and note
first that  $S^\perp$ is also $\mathcal U_n$-invariant. Then
\[
\text{$(Id-\pi_S)(f)\circ L_V \in S^\perp$
\,\,\,\, and \,\,\,\,$f\circ L_V=\pi_S(f)\circ L_V+(Id-\pi_S)(f)\circ L_V$\,,}
\]
and hence
\begin{align}\label{eq: pi_S commutes with L_V}
\pi_S(f\circ L_V)=\pi_S(\pi_S(f)\circ L_V)+\pi_S((Id-\pi_S)(f)\circ L_V)=\pi_S(f)\circ L_V\,.
\end{align}
Then
\[
\big\langle f, K_{VU}^S\big\rangle_{L_2}=\pi_S(f)(VU)=\pi_S(f)\circ L_V(U)=\pi_S(f\circ L_V)(U)\,,
\]
and thus  by $(i)$
\[
\big\langle f, K_{VU}^S\big\rangle_{L_2}= \big\langle f\circ L_V, K_{U}^S\big\rangle_{L_2} = \big\langle C_{L_V}f , K_{U}^S\big\rangle_{L_2}
= \big\langle  f , C_{L_{V^{-1}}} K_{U}^S\big\rangle_{L_2} = \big\langle  f ,  K_{U}^S \circ L_{V^{-1}}\big\rangle_{L_2}\,.
\]
Since $f\in L_2(\mathcal U_n)$ was chosen arbitrarily, we obtain that  $K_{VU}^S=K_{U}^S \circ L_{V^{-1}}.$ The other identity follows similarly.

\noindent
$(iv)$ Let $V\in\mathcal U_n, $ then
\begin{eqnarray*}
K_V^S(V)=\big\langle K_V^S, K_V^S\big\rangle_{L_2}=\big\langle K_{Id}^S\circ L_{V^{-1}}, K_V^S\big\rangle_{L_2}=\big\langle K_{Id}^S, K_V^S\circ L_V\big\rangle_{L_2}=\big\langle K_{Id}^S, K_{Id}^S\big\rangle_{L_2}=K_{Id}^S(Id)>0\,. \qedhere
\end{eqnarray*}
\end{proof}

\smallskip

It remains to prove  Theorem~\ref{newperspective}. Defining
\begin{equation}\label{new}
  \mathrm{t}_S  := K^S_{Id}\,,
\end{equation}
this proof is  in fact  a straight forward consequence of the preceding lemma.
But before we do this, we collect  two elementary properties of the kernel $\mathrm{t}_S $.

\begin{remark} \label{properties}
Let $S$ be a $\mathcal U_n$-invariant subspace of  $C(\mathcal{U}_n)$, which is closed in  $L_2(\mathcal U_n)$. Then   $\mathrm{t}_S  = K^S_{Id}$
satisfies:
  \begin{itemize}
    \item
  $\mathrm{t}_S(V^\ast) = \overline{\mathrm{t}_S(V)}$  for all $V \in \mathcal{U}_n$\,,
  \item
$\mathrm{t}_S(V^{*}UV) = \mathrm{t}_S(U)$ for all $U,V \in \mathcal{U}_n$, that is, $\mathrm{t}_S$ is a so-called class function.
\end{itemize}
\end{remark}

Indeed, for the first equality note that
\[
\mathrm{t}_S(V^\ast) = (K_{Id}^S \circ L_{V^{-1}})(Id) = K^S_V(Id) = \overline{ K^{Id}_V(S)} = \overline{\mathrm{t}_S(V)}\,,
\]
and together with this we get
\begin{align*}
   \mathrm{t}_S(V^{-1}UV)
   &
   =  \overline{\mathrm{t}_S(V^{-1} U^\ast V)} = \overline{(K_{Id}^S \circ L_{V^{-1}})(U^{\ast}V)}
    \\&
    =  \overline{K_{V}^S (U^{\ast}V)} =  K_{U^{\ast}V}^S (V) =  K_{Id}^S \circ R_{V^{-1}U}(V) =\mathrm{t}_S(U).
\end{align*}

\begin{proof}[Proof of Theorem~\ref{newperspective}]
 By Lemma~\ref{theorem kernel 12.2.5} for all $U \in \mathcal{U}_n$ and $f \in L_2(\mathcal{U}_n)$
\begin{align*}
(\pi_Sf)(U)
&
=\int_{\mathcal{U}_n}f(V)\overline{K_U^S(V)}dV
\\&
=\int_{\mathcal{U}_n}f(V)K_V^S(U)dV
\\&
=\int_{\mathcal{U}_n}f(V) K^S_{Id}(UV^{-1})dV
=\int_{\mathcal{U}_n}f(V) \mathrm{t}_S(UV^\ast)dV = (f \ast \mathrm{t}_S)(U)\,,
 \end{align*}
which proves $(i)$. Statement $(ii)$ was already shown in \eqref{eq: pi_S commutes with L_V}, and it remains to check $(iii)$.
Obviously, we have that
\[
\|\pi_S: C(\mathcal{U}_n) \to S\| =  \sup_{U \in \mathcal{U}_n} \int_{\mathcal{U}_n} |\mathrm{t}_S(UV^\ast)| dV\,,
\]
and for every $U \in \mathcal{U}_n$ by Remark~\ref{properties}
\[
 \int_{\mathcal{U}_n} |\mathrm{t}_S(UV^\ast)| dV = \int_{\mathcal{U}_n} |\mathrm{t}_S(V^\ast)| dV
 =\int_{\mathcal{U}_n} |\mathrm{t}_S(V)| dV\,.
\]
This completes the argument.
\end{proof}

\subsection{Accessibility}
  Let $S$  be $\mathcal U_n$-invariant subspace of  $C(\mathcal{U}_n)$, which is closed in  $L_2(\mathcal U_n)$. Then $S$ is called   \emph{accessible}
  if every projection $Q$  on $C(\mathcal U_n)$ onto $S$, which  commutes with the action of
$\mathcal U_n \times \mathcal U_n $
on $C(\mathcal U_n)$, equals  $\pi_S\restrict{C(\mathcal U_n)}$.

\begin{theorem} \label{abstract}
  Let $S$ be a $\mathcal U_n$-invariant and  accessible  subspace of  $C(\mathcal{U}_n)$, which is closed in  $L_2(\mathcal U_n)$. Then
  \[
\boldsymbol{\lambda}(S) =  \big\|\pi_S:C(\mathcal U_n) \to S\big\|=\int_{\mathcal{U}_n }|\mathrm{t}_S(V)| dV \,.
\]
\end{theorem}

  \begin{proof}
  The proof is an immediate consequence of  Rudin's Theorem~\ref{rudy} and the assumptions on $S$, taking into account that we
  know $(ii)$ and $(iii)$ from Theorem~\ref{newperspective} as well as \eqref{essentialpoint}.
        \end{proof}

We say that
 a $\mathcal U_n$-invariant subspace $S$ of  $C(\mathcal{U}_n)$, which is closed in  $L_2(\mathcal U_n)$, is  \emph{strongly accessible}, whenever  every $f \in S$ for which $ f(VUV^{*}) =f(U)$ for all
$U,V~\in~\mathcal{U}_n$, is a scalar  multiple of  $\mathrm{t}_S$. In other words, every class function in $S$ is a multiple of  $\mathrm{t}_S$\,.

As the name in the previous  definition suggests, we have the following key result.

\begin{proposition} \label{prop 12.2.7 for C(U)}
  Let $S$ be a $\mathcal U_n$-invariant subspace of $C(\mathcal{U}_n)$, which is closed in  $L_2(\mathcal U_n)$. Then $S$ is  accessible whenever it is strongly accessible.
\end{proposition}

The proof requires the next statement.

\begin{lemma}\label{prop 12.2.7 for L2(U)}
Let $H$ and $S$ be  $\mathcal U_n$-invariant  subspaces of  $C(\mathcal{U}_n)$,  which are both closed in  $L_2(\mathcal U_n)$.
Then, if $S$ is strongly accessible,  every operator $T~:~H\to  S$ that commutes with the action of $\mathcal U_n \times \mathcal U_n$
on $C(\mathcal{U}_n)$, is a scalar multiple of   
$\pi_{S}\restrict{H}$.

Moreover, if $H$ is  orthogonal to $S$ and $Q$ is  a projection on $H \bigoplus S$ onto $S$ that commutes with the action of $\mathcal U_n \times \mathcal U_n$
on $C(\mathcal{U}_n)$,  then $Q=\pi_{S}\restrict{H \bigoplus S}.$
\end{lemma}

\begin{proof}
By the assumption on $T$
 and Lemma~\ref{theorem kernel 12.2.5}, (iii) for every $V\in\mathcal U_n$,
\[
\big(C_{L_V}\circ C_{R_{V^{-1}}}\big) (T  \mathrm{t}_H)=T\big( (C_{L_V}\circ C_{R_{V^{-1}}}) \mathrm{t}_H\big) =T t_H\,.
\]
This implies that  $(T \mathrm{t}_H)(V^{\ast}UV)= (T\mathrm{t}_H)(U)$ for all $U,V\in\mathcal U_n$.
Since $S$ is strongly accessible, we have that $T  \mathrm{t}_H = \gamma   \mathrm{t}_S$ for some $\gamma  \in \mathbb{C}$. But from Theorem~\ref{newperspective} we know that for all $h \in H$
\[
h= \pi_H h = h \ast \mathrm{t}_H\,,
\]
and hence
\[
Th  = h \ast T\mathrm{t}_H = \gamma  h \ast \mathrm{t}_S = \gamma  \pi_S h\,.
\]
 To see the second assertion, note  that by  the first part of the lemma  we have $Q\restrict{H}=\gamma \,\pi_S\restrict {H}$ for some
$\gamma  \in \mathbb{C}$. But since by assumption $H\subset S^\perp$, this implies $Q\restrict{H}= 0=\pi_{S}\restrict {H}$.
 On the other hand,  since $Q$ is a~projection onto $S$, we see that  $Q\restrict{S}=Id_{S}=\pi_S\restrict {S}$, which finishes
 the proof.
\end{proof}

We now give a
\begin{proof}[Proof of Proposition \ref{prop 12.2.7 for C(U)}]
Let $Q$  be a projection on $C(\mathcal U_n)$ onto $S$, which  commutes with the action of
$\mathcal U_n \times \mathcal U_n $
on $C(\mathcal U_n)$.
By  Theorem~\ref{den-matrix}, it suffices to show that for each pair
$(p,q)~\in~\mathbb{N}_0~\times~\mathbb{N}_0 $
\[
Q\restrict{\mathfrak H_{(p,q)}} = \pi_S\restrict{\mathfrak H_{(p,q)}}\,.
\]
 Given  such pair $(p,q)$, we define the subspace
\[
 H := \big\{f-\pi_Sf\,:\,\,f\in \mathfrak H_{(p,q)}\big\} \subset C(\mathcal U_n).
\]
Then $H$ is  $\mathcal U_n$-invariant; indeed, by Theorem~\ref{newperspective}, (ii)
and the fact that $\mathfrak H_{(p,q)}$ is $\mathcal U_n$-invariant (proved in Lemma~\ref{ludo200}), for every   $f\in{\mathfrak H_{(p,q)}}$
and $U \in \mathcal U_n$ we have
\[
(f-\pi_S f)\circ L_U=f\circ L_U -\pi_S f\circ L_U= f\circ L_U -\pi_S(f\circ L_U)\in H\,,
\]
and the  invariance under right multiplication follows similarly.
 Since $H\perp S$ and $Q$ commutes with the action of $\mathcal U_n \times \mathcal U_n$
 on $C(\mathcal U_n)$,
 Lemma~\ref{prop 12.2.7 for L2(U)} (the second part applied to the restriction  of $Q$ to $H \oplus S$) shows  that
 $$Q\restrict{H \oplus S}=  \pi_S\restrict{H \oplus S}\,,$$
  so  in particular
 $
     Q\restrict{H}=  \pi_S\restrict{H} = 0\,.
     $
But then for every $f \in \mathfrak{H}_{(p,q)}(\mathcal U_n)$
\begin{equation*}
   Q(f) = Q(f-\pi_S f)+ Q(\pi_Sf) =   \pi_Sf\,,
  \end{equation*}
  which completes the argument.
  \end{proof}

\subsection{The special case $S= \mathfrak{H}_{(1,0)}(\mathcal{U}_n)$} \label{special}
Recall from Section~\ref{Unitarily invariant subspaces} the definition of the $\mathcal{U}_n$-invariant subspace $\mathfrak{H}_{(1,0)}(\mathcal{U}_n)$
of $C(\mathcal{U}_n)$ of all polynomials $f \in C(\mathcal{U}_n)$ of the form
\[
f(U) = \sum_{1 \leq i,j \leq n} c_{i,j} u_{i,j}\,,
\]
where $U = (u_{i,j})_{1 \leq i,j \leq n}\in \mathcal{U}_n$\,.

In Theorem~\ref{newperspective} we showed that the orthogonal projection
$\pi_{(1,0)} = \pi_{\mathfrak{H}_{(1,0)}(\mathcal{U}_n)}$ on $L_2(\mathcal{U}_n)$ onto $\mathfrak{H}_{(1,0)}(\mathcal{U}_n)$
is a convolution operator with respect to the kernel $t_{(1,0)} = t_{\mathfrak{H}_{(1,0)}(\mathcal{U}_n)}$. We need an alternative description of this projection in terms of the canonical orthonormal basis of $\mathfrak{H}_{(1,0)}(\mathcal{U}_n)$.

By~\eqref{intformA}  the collection  of all  normalized functions $\sqrt{n}\,e_{ij}, \, 1 \leq i,j \leq n$ forms an orthonormal
system in $L_2(\mathcal{U}_n)$, and hence an orthonormal basis of  $\mathfrak H_{(1,0)}(\mathcal U_n)$
considered as a subspace of $L_2(\mathcal U_n)$.
Consequently,   for each $f\in L_2(\mathcal U_n)$
\begin{equation}\label{orthonormalrep}
  \pi_{(1,0)}(f)=\sum_{1 \leq i,j \leq n}\big\langle f, \sqrt n e_{ij} \big\rangle_{L_2}\, \sqrt n e_{ij}= n \sum_{1 \leq i,j \leq n}\big\langle f,  e_{ij} \big\rangle_{L_2}\,\, e_{ij}\,,
\end{equation}
where
$e_{ij} \in \mathfrak H_{(1,0)}(\mathcal U_n)$
is defined by
$e_{ij}(U) = u_{i,j}$ for $U \in \mathcal{U}_n$\,.

Comparing the  two
representations of $\pi_{(1,0)}$  we now have, leads to the following

\begin{proposition} \label{specialI}
For each  $n \in \mathbb{N}$ we  have  $\mathrm{t}_{(1,0)} = n \,\mathrm{tr}$, and moreover
\[
\text{$\pi_{(1,0)}f = n \,\,(f \ast \mathrm{tr})$\,,\quad\,  $f\in L_2(\mathcal{U}_n)$}
\]
and
\[
\text{$\|\pi_{(1,0)}: C(\mathcal{U}_n) \to \mathfrak{H}_{(1,0)}(\mathcal{U}_n )\|  = n\int_{\mathcal{U}_n }|\mathrm{tr}(V)| dV $\,.}
\]
\end{proposition}

\begin{proof}
  To check the equality $\mathrm{t}_{(1,0)} = n \,\mathrm{tr}$, recall that by Lemma~\ref{theorem kernel 12.2.5},~(i) and
    the definition of $\mathrm{t}_{(1,0)}$ from~\eqref{new},  for all $f\in L_2(\mathcal{U}_n)$ one gets
    \[
(\pi_{(1,0)}f)(Id)=\big\langle f, \mathrm{t}_{(1,0)} \big\rangle_{L_2}\,.
\]
On the other hand, by~\eqref{orthonormalrep} for all $f\in L_2(\mathcal{U}_n)$,
\begin{align*}
  (\pi_{(1,0)}f)(Id)
    = n \sum_{ i,j}\big\langle f,  e_{ij} \big\rangle_{L_2}\,\, e_{ij}(Id)
    =
    n \sum_{ i}\big\langle f,  e_{ii} \big\rangle_{L_2} = n \big\langle f,   \mathrm{tr} \big\rangle_{L_2}\, \,,
  \end{align*}
  which together with the preceding equality is  what we were looking for. To deduce the second and third claim, is then immediate from  of Theorem~\ref{newperspective}, $(iii)$.
\end{proof}

\begin{proposition}\label{specialII}
$\mathfrak{H}_{(1,0)}(\mathcal{U}_n)$ is a strongly  accessible   $\mathcal U_n$-invariant subspace of  $C(\mathcal{U}_n)$.
\end{proposition}

\begin{proof}
 Take
$f =\sum_{1 \leq i,j \leq n} c_{i,j} e_{i,j}\in \mathfrak H_{(1,0)}(\mathcal U_n)$
such that $f(V^{-1}UV)=f(U)$ for every $U,V\in\mathcal U_n$.  Clearly,  $f$ can be considered  as a linear functional
on $M_n(\mathbb{C})$. This implies that there exists $A\in M_n(\mathbb{C})$ such that $f(U)=\mathrm{tr}(AU)$ for all $U\in M_n(\mathbb{C})$.
Then  from the assumption on $f$,  it follows that for all $U,V\in\mathcal U_n$
\[
\mathrm{tr}(AU)= f(U)= f(V^{-1}UV) =\mathrm{tr}(AV^{-1}UV)=\mathrm{tr}(VAV^{-1}U).
\]
Combining this with the fact that any  matrix in $M_n(\mathbb{C})$ is a linear combination of unitary matrices, we deduce that $A=VAV^{-1}$ for every $V\in\mathcal U_n$, and so $A$ commutes with all matrices in $M_n(\mathbb{C})$. This  implies that $A=\gamma  Id$ for some
$\gamma \in \mathbb C$, and hence  as desired   $f=\gamma \,\mathrm{tr}$.
\end{proof}

A comment is in order: if $p+q>1$ then  $g_1(A):=\mathrm{tr}(A^p(A^*)^q) $ and $g_2(A):= \mathrm{tr}(A)^p\mathrm{tr}(A^*)^q$  are different class functions. Thus, in this case, $\mathfrak H_{(p,q)}(\mathcal U_n)$ is not strongly accessible.

\smallskip
The following result identifies $\mathfrak H_{(1,0)}(\mathcal U_n)$ with the trace class $\mathcal S_1(n)$.

\begin{proposition}\label{linco}
 The space $\mathfrak H_{(1,0)}(\mathcal U_n)$ is isometrically isomorphic to $\mathcal S_1(n)$. More precisely,
 \begin{equation}\label{eq: S1 to H(1,0)}
 \mathcal S_1(n) \to \mathfrak H_{(1,0)}(\mathcal U_n)\,, \,\,\,\,\, A \mapsto [f:U \mapsto \mathrm{tr}(AU)]
 \end{equation}
 is an isometry onto.
\end{proposition}

\begin{proof}
Obviously, the mapping in~\eqref{eq: S1 to H(1,0)} is a linear bijection. Indeed, as a linear space $\mathcal S_1(n)$ equals $M_n(\mathbb{C})$, and
$\mathfrak H_{(1,0)}(\mathcal U_n)$ equals the algebraic dual $M_n(\mathbb{C})^{\times}$ of $M_n(\mathbb{C})$. Moreover, it  is well-known that the
mapping $A \mapsto [f:U \mapsto \mathrm{tr}(AU)]$ identifies $M_n(\mathbb{C})$ and $M_n(\mathbb{C})^{\times}$. So it remains to prove
that the mapping in~\eqref{eq: S1 to H(1,0)}  is isometric.
 To prove  that, we use a result
of Nelson \cite{nelson1961distinguished} (see also \cite[Theorem 1]{harris1997holomorphic}) showing that for any
complex-valued   function $f$, which is continuous on the closed  and  analytic
on the open unit ball of ${\mathcal L(\ell_2^n)}$, we have
$
    \sup_{\|T\| \leq  1  }|f(T)|=\sup_{U\in\mathcal U_n}|f(U)| \,.
$
But then  by~\eqref{dualitytr} for every $A \in \mathcal S_1(n)$
\[
\|A\|_1 = \sup_{\|T\| \leq 1  }|\mathrm{tr}(AT)| =\sup_{U\in\mathcal U_n} |\mathrm{tr}(AU)|\,,
\]
completing the argument.
\end{proof}

\subsection{Proof of the main result} \label{proof of the main result}
We start by presenting the

\begin{proof}[Proof of the integral formula from~\eqref{intgralformula}]
We use the identification from  Proposition~\ref{linco}, and combine  it with Proposition~\ref{specialII} and  Theorem~\ref{abstract}.
Then  Proposition~\ref{specialI} completes  the argument.
  \end{proof}

Now we deal with the limit formula from  \eqref{cor: lim proj constant trace}. For this we  need to recall  some well-known results from probability theory (for more on this see \cite{billingsley2013convergence}). We
are going to  use that, given any sequence $(Y_n)$ of random variables,  which converges in distribution to the  random variable $Y$,  and any continuous real-valued function $f$,
the sequence $\big(f(Y_n)\big)$ converges in distribution to $f(Y)$.
Recall also that a~sequence $(Y_n)_{n }$ of random variables  is said to be  uniformly integrable whenever $$\lim_{a \to  \infty} \sup_{n \ge 1} \int_{|Y_n| \ge a } |Y_n| d P = 0\,.$$
Uniform integrability will be useful for us due to the fact (see for example \cite[Theorem 3.5]{billingsley2013convergence})
that  if $(Y_n)_{n}$ is a uniformly integrable sequence of random variables and $Y_n \overset{D}{\longrightarrow} Y$, then $Y$ is integrable and \begin{align}\label{thm: unif int + conv in dist implies conv in mean}
\mathbb{E} (Y_n) \to \mathbb{E}(Y)\,.
\end{align}

To check uniform integrability we cite a standard criterion.

\begin{remark}\label{rem: bounded moment implies unif int}
If  $\sup_{n} \mathbb{E}( |Y_n|^{1+\varepsilon}) \le C$ for some $\varepsilon, C>0$, then $(Y_n)_{n}$ is uniformly integrable; indeed,
this is a consequence of
\[
\lim_{a \to  \infty} \sup_{n \ge 1} \int_{|Y_n| \ge a } |Y_n| d P \le \lim_{a \to  \infty} \frac{1}{a^\varepsilon} C\,.
\]
\end{remark}

We are now ready to provide the

\begin{proof}[Proof of the limit formula from \eqref{cor: lim proj constant trace}]

Consider the sequence $\big(\mathrm{tr}(U(n))\big)$  of random variables on $\mathcal{U}_n$, where $U(n)$ is a unitary matrix uniformly Haar distributed. Then, by \cite[Corollary 2.4]{johansson1997random} (see also \cite{diaconis1994eigenvalues} or \cite[Problem 8.5.5]{pastur2011eigenvalue}), the previous sequence converges in distribution to the standard Gaussian complex random variable
$\pmb{\gamma}$. Indeed,
 the random variables
 $\sqrt{2} Re[\mathrm{tr}(U(n))]$ and $\sqrt{2} Im[\mathrm{tr}(U(n))]$
 converge in distribution to a standard real Gaussian random variable.

Thus, the sequence $\big(\sqrt{2}|\,\mathrm{tr}(U(n))|\big)$ of
random variables on $\mathcal{U}_n$ converges in distribution to a Rayleigh random variable.
Moreover, since as mentioned in~\eqref{intformB}, for each~$n$
\[
\mathbb{E}\big(|\mathrm{tr}(U(n))|^2\big)=
\int_{\mathcal U_n} |\mathrm{tr}(V)|^2\,dV  =  1\,,
\]
  the sequence of random variables  $\mathrm{tr}(U(n))$ by  Remark \ref{rem: bounded moment
implies unif int} is  uniformly integrable . Consequently,  we deduce from  \eqref{thm: unif int + conv in dist implies conv in mean}
that $\big(\mathbb{E}(\sqrt{2}| \,\mathrm{tr}(U(n))|)\big)$ converges to the expectation of a Rayleigh random variable. That is,
\[
\lim_{n\to \infty}\, \mathbb{E}\big(\sqrt{2}\,|\mathrm{tr}(U(n))|\big)\to\sqrt{\frac{\pi}{2}}\,.
\]
Using~\eqref{intgralformula}, we arrive at
\[
\lim_{n\to \infty} \frac{1}{n}\boldsymbol{\lambda}(\mathcal S_1(n)))=
\frac{1}{\sqrt 2}\,\lim_{n\to \infty} \mathbb E\big(\sqrt{2}\,|\mathrm{tr}(U(n))|\big) = \frac{\sqrt{\pi}}{2}\,,
\]
which completes the proof.
\end{proof}

\subsection{Another examples}
In this final subsection  we give some other examples where the theory developed to reach our main objective (Theorem~\ref{mainU}) could be applied.

The first result  shows that examples of accessible $\mathcal U_n$-invariant subspaces come in pairs. To see this we define the
linear and isometric bijection
\[
\phi: C(\mathcal{U}_n) \to C(\mathcal{U}_n), \,\,f \mapsto [U \mapsto f(U^\ast)]\,.
\]
For any subspace $S$ in $C(\mathcal{U}_n)$, we write $S_\ast: = \phi S$. As a first  example we mention that isometrically
\[
\big(\mathfrak{H}_{(1,0)}(\mathcal{U}_n)\big)_\ast=\phi \big(\mathfrak{H}_{(1,0)}(\mathcal{U}_n)\big) = \mathfrak{H}_{(0,1)}(\mathcal{U}_n)\,.
\]

\begin{proposition}  \label{relation with the conjungate space}
 Let $S$ be a $\mathcal U_n$-invariant  subspace of  $C(\mathcal{U}_n)$,  which is closed in  $L_2(\mathcal U_n)$. Then~$S_\ast$ is
 $\mathcal U_n$-invariant and
 $
 \mathrm{t}_{S_\ast} = \overline{\mathrm{t}_{S}}\,.
 $
  Moreover, $S$ is strongly accessible $($resp., accesible$)$ if and only if $S_\ast$ is  strongly accessible $($resp., accessible$)$.
  \end{proposition}

 \begin{proof}
 Obviously, $S_\ast$ is $\mathcal U_n$-invariant. In order to show that $\mathrm{t}_{S_\ast} = \overline{\mathrm{t}_{S}}$
 note first that $\pi_{S_\ast} = \phi \circ\pi_{S}\circ \phi$. Then for every $f \in L_2(\mathcal U_n)$
 and $U \in \mathcal U_n$, it follows
 by Theorem~\ref{newperspective}
 and Remark~\ref{properties} that
  \begin{align*}
   \big(\pi_{S_\ast}f\big)(U)
   &
   =
   \big((\pi_{S} \phi f)\big)(U^\ast)
    \\&
   = \big(\phi f \ast \mathrm{t}_{S}\big)(U^\ast)
      = \int_{\mathcal U_n} f(V^\ast)\mathrm{t}_{S}(U^\ast V^\ast) dV
   \\&
     = \int_{\mathcal U_n} f(V^\ast)\overline{\mathrm{t}_{S}}(VU) dV
          = \int_{\mathcal U_n} f(V)\overline{\mathrm{t}_{S}}(V^\ast U) dV
          \\&
          = \int_{\mathcal U_n} f(V)\overline{\mathrm{t}_{S}}( U V^\ast) dV = \big(f \ast \overline{\mathrm{t}_{S}}\big)(U)\,,
 \end{align*}
 which by the uniqueness of $\mathrm{t}_{S_\ast}$ leads to the claim. Let us turn to the 'moreover part'. It is immediate
 that strong accessibility of $S$ is equivalent to strong accessibility of $S_\ast$. So let us assume that
 $S$ is  accessible, and show that then $S_\ast$ is accessible. Take any
 projection $Q: C(\mathcal U_n) \to S_\ast$ which commutes with the action of $\mathcal U_n \times \mathcal U_n$
 on  $C(\mathcal U_n)$. Since $\phi \circ C_{L_V} = R_{V^\ast} \circ \phi$ and
 $\phi \circ C_{R_V} = L_{V^\ast} \circ \phi$ for all $V \in\mathcal U_n $, the projection $\phi \circ Q\circ \phi$ onto $S$
 commutes with the action of $\mathcal U_n \times \mathcal U_n$ on  $C(\mathcal U_n)$, and hence by assumption $\phi \circ Q\circ \phi = \pi_S$.
 But then clearly $Q = \phi \circ\pi_{S}\circ \phi = \pi_{S_\ast}$, the conclusion.
       \end{proof}

 Note that, in particular, $\mathfrak{H}_{(0,1)}(\mathcal{U}_n)$ is $\mathcal U_n$-invariant and  accessible,
 and $\mathrm{t}_{(0,1)} = \overline{\mathrm{tr}}$ and therefore
 by Theorem~\ref{abstract}
\begin{equation*}
   \boldsymbol{\lambda}\big(\mathfrak{H}_{(0,1)}(\mathcal{U}_n)\big) = \big\|\pi_{(0,1)}:C(\mathcal U_n) \to  \mathfrak{H}_{(0,1)}(\mathcal{U}_n)\big\| = n \int_{\mathcal{U}_n }|\mathrm{tr} (V)| d V\,.
   \end{equation*}
Also,
\begin{equation*}
  \lim_{n\to \infty} \, \frac{\boldsymbol{\lambda}\big(\mathfrak{H}_{(0,1)}(\mathcal{U}_n)\big)}{n}=\frac{\sqrt \pi}{2}\,.
\end{equation*}
Of course, this is also a simple consequence
of Theorem~\ref{mainU} using that $\phi$
identifies $\mathfrak{H}_{(0,1)}(\mathcal{U}_n)$ and
$\mathfrak{H}_{(1,0)}(\mathcal{U}_n)$ isometrically.

We continue with
another simple stability property of accessible subspaces.

 \begin{proposition}
 Let $S_1$  and $S_2$  be accessible, $\mathcal U_n$-invariant subspaces of $C(\mathcal U_n)$,  which  in  $L_2(\mathcal U_n)$
 are closed and orthogonal. Then, $S_1 \oplus S_2$ is  accessible and $\mathcal U_n$-invariant, and moreover $t_{S_1 \oplus S_2} =  \mathrm{t}_{S_1} +  \mathrm{t}_{S_2}$.
 Consequently,
\begin{equation}\label{integral formula sum}
\boldsymbol{\lambda}(S_1 \oplus S_2) = \big\|\pi_{S_1} + \pi_{S_2}:C(\mathcal U_n) \to S_1 \oplus S_2  \big\|
=
 \int_{\mathcal{U}_n }|\mathrm{t}_{S_1}(V) +  \mathrm{t}_{S_2}(V)| dV
\,.
\end{equation}

 \end{proposition}

\begin{proof}
That $S_1 \oplus S_2$ is $\mathcal U_n$-invariant is straightforward.
Note that $\pi_{S_1 \oplus S_2}  = \pi_{S_1} + \pi_{S_2}$
is the orthogonal projection on $L_2(\mathcal{U}_n)$ onto~$S_1 \oplus S_2$. Then by
Theorem~\ref{newperspective}
for all $f \in L_2(\mathcal{U}_n)$ we have
\[
\pi_{S_1 \oplus S_2}f = \pi_{S_1}f + \pi_{S_2}f = f \ast \mathrm{t}_{S_1} + f \ast \mathrm{t}_{S_2} = f \ast (\mathrm{t}_{S_1} +\mathrm{t}_{S_2})\,.
\]
Hence by the uniqueness of $t_{S_1 \oplus S_2}$ we get
\[
\mathrm{t}_{S_1 \oplus S_2} = \mathrm{t}_{S_1} +\mathrm{t}_{S_2} \,.
\]
Let us now show that $S_1 \oplus S_2$ is accessible. So let $Q$ be a projection on
$C(\mathcal{U}_n)$ onto $S_1 \oplus S_2$ which commutes with the action of $\mathcal{U}_n\times \mathcal{U}_n$ on
$C(\mathcal{U}_n)$. We claim that $Q = \pi_{S_1 \oplus S_2}$. Indeed, consider the two projections
\[
\text{$Q_{S_1} = \pi_{S_1}\circ   Q $ \quad and \quad $Q_{S_2} = \pi_{S_2}\circ   Q $}
\]
on $C(\mathcal{U}_n)$ onto $S_1$ and ${S_2}$, respectively.
Since $\pi_{S_1}$ and $\pi_{S_2}$ both commute with the action of $\mathcal{U}_n\times \mathcal{U}_n$ on
$C(\mathcal{U}_n)$, also $Q_{S_1}$ and $Q_{S_2}$ do. Then by the accessibility of $S_1$ and $S_2$ we see that
\[
\text{$Q_{S_1} = \pi_{S_1}$ \quad and \quad $Q_{S_2} = \pi_{S_2}$}\,,
\]
and hence for all $f \in C(\mathcal{U}_n)$ as desired
\[
Qf = \pi_{S_1}(Qf)  + \pi_{S_2}(Qf) = \pi_{S_1}f + \pi_{S_2}f = \pi_{S_1 \oplus S_2} f\,.
\]
To conclude the proof just note that \eqref{integral formula sum} is then a direct consequence of Theorem~\ref{abstract}.
\end{proof}

Combining the  previous two propositions we obtain

\begin{corollary} \label{coro: sum 1,0 + 0,1}
For each $n \in \mathbb{N}$,
\begin{align*}
   \boldsymbol{\lambda}\big(\mathfrak{H}_{(1,0)}(\mathcal{U}_n) \oplus \mathfrak{H}_{(0,1)}(\mathcal{U}_n)\big) & = \big\|\pi_{(1,0)} \oplus \pi_{(0,1)} :C(\mathcal U_n) \to  \mathfrak{H}_{(1,0)}(\mathcal{U}_n) \oplus \mathfrak{H}_{(0,1)}(\mathcal{U}_n)\big) \big\| \\
   &= 2n \int_{\mathcal{U}_n }|Re(\mathrm{tr} (V))| d V\,.
   \end{align*}
Moreover,
\begin{equation} \label{limit suma 1,0+0,1}
  \lim_{n\to \infty} \, \frac{\boldsymbol{\lambda}\big(\mathfrak{H}_{(1,0)}(\mathcal{U}_n) \oplus \mathfrak{H}_{(0,1)}(\mathcal{U}_n)\big)}{\sqrt{2} n}=\sqrt{\frac{2}{\pi}}\,.
\end{equation}
\end{corollary}

Before giving a proof of this, we mention that the denominator of the fraction above (so $\sqrt{2}  n$) is exactly the square root of the dimension of the sum space $\mathfrak{H}_{(1,0)}(\mathcal{U}_n) \oplus \mathfrak{H}_{(0,1)}(\mathcal{U}_n)$.
\begin{proof}[Proof of Corollary \ref{coro: sum 1,0 + 0,1}]

We only have to prove \eqref{limit suma 1,0+0,1} since the integral formula for the projection constant follows directly from \eqref{integral formula sum} and Proposition \ref{specialI}.

We repeat an argument similar to the proof of \eqref{cor: lim proj constant trace}.
We know, by \cite[Corollary 2.4]{johansson1997random}, that the sequence $\big(\sqrt{2} Re[\mathrm{tr}(U(n))]\big)$ of random  variables converges in distribution to a standard real Gaussian random variable $g$.
In particular, $\big(\sqrt{2} \vert  Re[\mathrm{tr}(U(n))] \vert \big)$
converges in distribution to $ \vert  g \vert$.
Note that the sequence $\big(\sqrt{2} \vert  Re[\mathrm{tr}(U(n))] \vert \big)$ is uniformly integrable. Indeed, $ \mathbb{E} (\vert Re[\mathrm{tr}(U(n))] \vert^2 ) \leq  \mathbb{E} (\vert \mathrm{tr} (U(n)) \vert^2) =1$
(see again  Remark \ref{rem: bounded moment
implies unif int}).
Thus by \eqref{thm: unif int + conv in dist implies conv in mean},
\begin{align*}
  \lim_{n \to \infty} \frac{\boldsymbol{\lambda}\big(\mathfrak{H}_{(1,0)}(\mathcal{U}_n) \oplus \mathfrak{H}_{(0,1)}(\mathcal{U}_n)\big)}{\sqrt{2}  n}
    = \lim_{n \to \infty} \mathbb{E} \big(\sqrt{2} \vert Re[\mathrm{tr}(U(n))] \vert \big)
   =   \mathbb{E} |g| = \frac{1}{\sqrt{2\pi}} \int_{\mathbb{R}} |x| e^{-\frac{{x}^2}{2}} dx = \sqrt{\frac{2}{\pi}}.
\end{align*}
This concludes the proof.
\end{proof}





\end{document}